\documentclass[11 pt]{amsart}
\usepackage{latexsym,amscd,amssymb, graphicx, amsthm}  
\usepackage{blkarray}
\usepackage{tikz-cd, xcolor}
\usetikzlibrary{calc}
\usepackage[margin=1in]{geometry}

\numberwithin{equation}{section}

\newtheorem{theorem}{Theorem}[section]
\newtheorem{proposition}[theorem]{Proposition}

\newtheorem{lemma}[theorem]{Lemma}
\newtheorem{conjecture}[theorem]{Conjecture}

\newtheorem{question}[theorem]{Question}

\definecolor{2purple}{RGB}{204,102,255}
\definecolor{3green}{RGB}{0,204,0}

\newtheorem{definition}[theorem]{Definition}
\theoremstyle{definition}

\newcommand{\symm}{{\mathfrak{S}}}

\newcommand{\CC}{{\mathbb {C}}}

\begin{document}

\title[A skein action embedding]{An embedding of the skein action on set partitions into the skein action on matchings}

\author{Jesse Kim}
\address
{Department of Mathematics \newline \indent
University of California, San Diego \newline \indent
La Jolla, CA, 92093-0112, USA}
\email{jvkim@ucsd.edu}

\begin{abstract}
Rhoades defined a skein action of the symmetric group on noncrossing set partitions which generalized an action of the symmetric group on matchings. The $\symm_n$-action on matchings is made possible via the Ptolemy relation, while the action on set partitions is defined in terms of a set of skein relations that generalize the Ptolemy relation. The skein action on noncrossing set partitions has seen applications to coinvariant theory and  coordinate rings of partial flag varieties. In this paper, we will show how Rhoades' $\symm_n$-module can be embedded into the $\symm_n$-module generated by matchings, thereby explaining how Rhoades' generalized skein relations all arise from the Ptolemy relation.
\end{abstract}

\maketitle

\section{Introduction}
This paper concerns an action of $\symm_n$ on the vector space spanned by the set of noncrossing set partitions of $[n] := \{1,2,\dots, n\}$ first defined by Rhoades \cite{Rhoades}. This action is defined in terms of three {\em skein relations}, the simplest of which is the {\em Ptolemy relation} shown below. 
\begin{center}

\begin{tikzpicture}[scale=0.5]

    \newcommand{\squa}
        {
        \foreach \i in {1,...,4}
                {
                \coordinate (p\i) at (90*\i + 45:\r);
                \filldraw(p\i) circle (\pointradius pt);
                }
        }
            
    \def\r{1}           
    \def\pointradius{3} 

    \foreach \a/\b/\c/\d [count=\shft from 0] in   
       {1/3/2/4,
        1/2/3/4,
        1/4/2/3}
        {
        \begin{scope}[xshift=100*\shft,yshift=0]
            \squa
            \foreach \i in {1,...,4}
                \draw[thick] (p\a) -- (p\b);
                \draw[thick] (p\c) -- (p\d);
            \squa
        \end{scope}
        }

	\node at (1.8,0) {$\mapsto$};
	\node at (5.2,0) {$+$};

    \end{tikzpicture}
\end{center}
The original motivation for defining this action was to give an algebraic proof of certain cyclic sieving results on noncrossing set partitions, first proven by Reiner, Stanton, and White \cite{RSW} and Pechenik \cite{P} via direct enumeration. Rhoades' action has since been found within coinvariant rings and coordinate rings of certain partial flag varieties\cite{me,PPS}, strengthening the claim that it is an action worth studying in its own right. Rhoades' action generalizes a similar action on the vector space with basis given by noncrossing matchings. This paper will show how Rhoades' action can be found within the action on noncrossing matchings and thereby explain how the three skein relations all arise from the Ptolemy relation.

More precisely, let $M(n)$ denote the set of all {\em matchings} of $[n]$, i.e. collections of disjoint size-two subsets of $[n]$. The symmetric group $\symm_n$ acts naturally on $M(n)$ as follows. If $\sigma \in \symm_n$ and $m = \{\{a_1,b_1\}, \dots, \{a_k, b_k\}\}$ is a matching, then
\begin{equation}
\sigma \star m = \{ \{\sigma(a_1), \sigma(b_1)\}, \dots, \{\sigma(a_k), \sigma(b_k)\}\}.
\end{equation}
We can extend this action to an action on $\mathbb{C}[M(n)]$, the $\mathbb{C}$-vector space with basis given by matchings of $[n]$. For our purposes, it will be more useful to consider a sign-twisted version of this action, where
\begin{equation}
\sigma \circ m = \mathrm{sign}(\sigma)\{ \{\sigma(a_1), \sigma(b_1)\}, \dots, \{\sigma(a_k), \sigma(b_k)\}\}.
\end{equation}

A matching is {\em noncrossing} if it does not contain two subsets $\{a,c\}$ and $\{b,d\}$ with $a<b<c<d$. The action on matchings does not descend to an action on $NCM(n)$, the set of all noncrosssing matchings of $[n]$, since permuting elements in a noncrossing matching could introduce crossings. However, we can linearize and define an action on $\mathbb{C}[NCM(n)]$, the $\mathbb{C}$-vector space with basis given by noncrossing matchings of
$[n]$. For any noncrossing matching $m$ and adjacent transposition $s_i = (i,i+1)$, define
\begin{equation}
s_i \cdot m = \begin{cases} s_i \circ m & s_i \circ m \textrm{ is noncrossing}\\ m + m' & \textrm{otherwise} \end{cases}.
\end{equation}
Here $\circ$ denotes the action on all matchings and $m'$ is the matching where the subsets of $m$ containing $i$ and $i+1$, call them $\{i,a\}$ and $\{i+1,b\}$ have been replaced with $\{i,i+1\}$ and $\{a,b\}$ and all other subsets remain the same.  In other words, $s_i \circ m$, $m$, and $m'$ form a trio of matchings that differ only in a Ptolemy relation. There exists an $\symm_n$-equivariant linear projection $p_M : \mathbb{C}[M(n)] \rightarrow \mathbb{C}[NCM(n)]$ given for any matching $m$ by 
\begin{equation}
m \mapsto w^{-1} \cdot (w \circ m),
\end{equation}
where $w$ is any permutation for which $w \circ m$ is noncrossing. The kernel of this projection is spanned by sums of matchings which differ only by a Ptolemy relation, i.e.
\begin{multline}
\{\{a_1,a_2\},\{a_3,a_4\},\{a_5,a_6\}, \dots, \{a_{2k-1}, a_{2k}\}\} \\+ \{\{a_1,a_3\},\{a_2,a_4\},\{a_5,a_6\}, \dots, \{a_{2k-1}, a_{2k}\}\} \\+ \{\{a_1,a_4\},\{a_2,a_3\},\{a_5,a_6\}, \dots, \{a_{2k-1}, a_{2k}\}\}
\end{multline}
for any distinct $a_1, \dots, a_{2k} \in [n]$. This projection can be thought of as a way to``resolve" crossings in a matching and obtain a sum of noncrossing matchings.

Analogously, let $\Pi(n)$ denote the set of all set partitions of $[n]$, and let $\mathbb{C}[\Pi(n)]$ be the $\mathbb{C}$-vector space with basis given by $\Pi(n)$. We can define an action of $\symm_n$ on $\mathbb{C}[\Pi(n)]$ analogous to the action on $\mathbb{C}[M(n)]$. A set partition is {\em noncrossing} if there do not exist distinct blocks $A$ and $B$ and elements $a,c \in A$, $b,d \in B$ with $a<b<c<d$. Let $NCP(n)$ denote the set of all noncrossing set partitions, and let $\mathbb{C}[NCP(n)]$ be the corresponding vector space. Rhoades \cite{Rhoades} defined an action of $\symm_n$ on $\mathbb{C}[NCP(n)]$ as follows. For any noncrossing set partition $\pi$ and adjacent transposition $s_i$,
\[
s_i \cdot \pi = \begin{cases} -\pi & i \textrm{ and } i+1 \textrm{ are in the same block of }\pi\\ -\pi'  & \textrm{at least one of } i \textrm{ and } i+1 \textrm{ is in a singleton block of } \pi \\ \sigma(\pi') & i \textrm{ and } i+1 \textrm{ are in different size 2 or larger blocks of }\pi \end{cases}
\]
where $\pi'$ is the set partition obtained by swapping which blocks $i$ and $i+1$ are in, and $\sigma$ is defined for any almost-noncrossing (i.e. the crossing can be removed by a single adjacent transposition) partition $\pi$ by $\sigma(\pi) = \pi + \pi_2 - \pi_3 - \pi_4$ where, if the crossing blocks in $\sigma$ are $\{i, a_1, \dots, a_k\}$ and $\{i+1, b_1, \dots, b_l\}$, then $\pi_2, \pi_3$ and $\pi_4$ are obtained from $\pi$ by replacing these blocks with
\begin{itemize}
\item $\{i, i+1\}$ and $\{a_1, \dots, a_k, b_1, \dots, b_l\}$ for $\pi_2$\\
\item $\{i,i+1, a_1, \dots, a_k\}$ and $\{b_1, \dots, b_l\}$ for $\pi_3$\\
\item $\{i,i+1, b_1, \dots, b_l\}$ and $\{a_1, \dots, b_k\}$ for $\pi_3$\\
\end{itemize}
when $k,l \geq 2$. If $k=1$ then $\pi_4 =0$ instead and if $l=1$ then $\pi_3 =0$ instead. The sum of partitions given by $\sigma(\pi)$ is best described with a picture, see Figure~\ref{skein relations}. The three possibilities (depending on whether $k,l \geq 2$) are the three skein relations mentioned earlier.
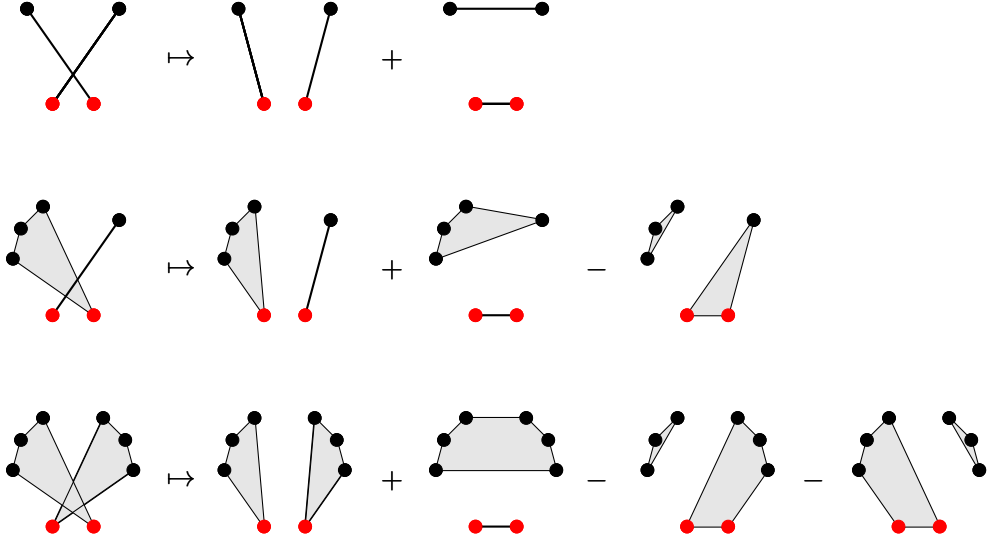
\begin{figure}
\begin{tikzpicture}[scale=0.8]

    \newcommand{\oneplusone}
        {
        \foreach \i in {1,2}
                {
                \coordinate (p\i) at (40*\i + 210:\r);
                \filldraw(p\i)[red] circle (\pointradius pt);
                }
         \foreach \j in {3,4}
                {
                \coordinate (p\j) at (100*\j + 100 :\r);
                \filldraw(p\j) circle (\pointradius pt);
                }
        }
            
    \def\r{1}           
    \def\pointradius{3} 

        \foreach \a/\b/\c/\d [count=\shft from 0] in   
       {1/3/2/4,
        1/4/2/3,
        1/2/3/4}
        {
        \begin{scope}[xshift=100*\shft,yshift=0]
            \oneplusone
            \foreach \i in {1,...,4}
                \draw[thick] (p\a) -- (p\b);
                \draw[thick] (p\c) -- (p\d);
                
                \oneplusone
            
        \end{scope}
        }

    \oneplusone

    \newcommand{\oneplussome}
        {
        \foreach \i in {1,2}
                {
                \coordinate (p\i) at (40*\i + 210:\r);
                \filldraw(p\i)[red] circle (\pointradius pt);
                }
         \foreach \j in {3}
                {
                \coordinate (p\j) at (100*\j + 100 :\r);
                \filldraw(p\j) circle (\pointradius pt);
                }
          \foreach \k in {4,5,6}
                {
                \coordinate (p\k) at (30*\k + 360 :\r);
                \filldraw(p\k) circle (\pointradius pt);
                }
        }
            
    \def\r{1}           
    \def\pointradius{3} 
    
     \foreach \a/\b/\c/\d/\e/\f [count=\shft from 0] in   
       {1/2/3/4/5/6,
        2/1/3/4/5/6,
        1/3/2/4/5/6}
        {
        \begin{scope}[xshift=100*\shft,yshift=-100]
           	 \oneplussome
	 	
               	 \draw[thick] (p\b) -- (p\d) -- (p\e) -- (p\f) -- cycle;
		 \fill[black!10] (p\b) -- (p\d) -- (p\e) -- (p\f) -- cycle;

                \draw[thick] (p\a) -- (p\c);
                
                \oneplussome
        \end{scope}
        }

         \foreach \a/\b/\c/\d/\e/\f [count=\shft from 0] in   
       {1/2/3/4/5/6}
        {
        \begin{scope}[xshift=100*\shft + 300,yshift=-100]
           	 \oneplussome
	 	
               	 \draw[thick] (p\a) -- (p\b) -- (p\c) -- cycle;
	 
	  \fill[black!10] (p\a) -- (p\b) -- (p\c)  -- cycle;
       
                \draw[thick] (p\d) -- (p\e) -- (p\f) -- cycle;
                
                 \fill[black!10] (p\d) -- (p\e) -- (p\f) -- cycle;
                 
                 \oneplussome
        \end{scope}
        }

    \tikzset{twopurple/.style={2purple,line width=8pt,rounded corners=2pt,cap=round}}
    \tikzset{threegreen/.style={3green,line width=8pt,rounded corners=2pt,cap=round,fill}}

    \newcommand{\someplussome}
        {
        \foreach \i in {1,2}
                {
                \coordinate (p\i) at (40*\i + 210:\r);
                \filldraw(p\i)[red] circle (\pointradius pt);
                }
         \foreach \j in {3,4,5}
                {
                \coordinate (p\j) at (30*\j + 270 :\r);
                \filldraw(p\j) circle (\pointradius pt);
                }
          \foreach \k in {6,7,8}
                {
                \coordinate (p\k) at (30*\k + 300 :\r);
                \filldraw(p\k) circle (\pointradius pt);
                }
        }
            
    \def\r{1}           
    \def\pointradius{3} 

         \foreach \a/\b/\c/\d/\e/\f/\g/\h [count=\shft from 0] in   
       {1/2/3/4/5/6/7/8,
       2/1/3/4/5/6/7/8}
        {
        \begin{scope}[xshift=100*\shft,yshift=-200]
           	 \someplussome
	 	
               	 \draw[thick] (p\a) -- (p\c) -- (p\d) -- (p\e) -- cycle;
		 \fill[black!10] (p\a) -- (p\c) -- (p\d) -- (p\e) -- cycle;
		 
		 \draw[thick] (p\b) -- (p\f) -- (p\g) -- (p\h) -- cycle;
		 \fill[black!10] (p\b) -- (p\f) -- (p\g) -- (p\h) -- cycle;
		 
		 \draw[thin] (p\a) -- (p\c);
		  \draw[thin] (p\a) -- (p\e);
		 
		 \someplussome

        \end{scope}
        }
        
                 \foreach \a/\b/\c/\d/\e/\f/\g/\h [count=\shft from 0] in   
       {1/2/3/4/5/6/7/8}
        {
        \begin{scope}[xshift=100*\shft + 200,yshift=-200]
           	 \someplussome
	 	
               	 \draw[thick]  (p\c) -- (p\d) -- (p\e) -- (p\f) -- (p\g) -- (p\h) -- cycle;
		 \fill[black!10] (p\c) -- (p\d) -- (p\e) -- (p\f) -- (p\g) -- (p\h) -- cycle;
		 
		 \draw[thick] (p\a) -- (p\b);

		 \someplussome

        \end{scope}
        }
        
                 \foreach \a/\b/\c/\d/\e/\f/\g/\h [count=\shft from 0] in   
       {1/2/3/4/5/6/7/8,
       1/2/6/7/8/3/4/5}
        {
        \begin{scope}[xshift=100*\shft + 300,yshift=-200]
           	 \someplussome
	 	
               	 \draw[thick] (p\a) -- (p\b) -- (p\c) -- (p\d) -- (p\e) -- cycle;
		 \fill[black!10] (p\a) -- (p\b) -- (p\c) -- (p\d) -- (p\e) -- cycle;
		 
		 \draw[thick]  (p\f) -- (p\g) -- (p\h) -- cycle;
		 \fill[black!10] (p\f) -- (p\g) -- (p\h) -- cycle;
		 
		 \someplussome

        \end{scope}
        
        \node at (1.8,-0.2) {$\mapsto$};
          \node at (5.3,-0.2) {$+$};

        \node at (1.8,-3.7) {$\mapsto$};
        \node at (5.3,-3.7) {$+$};
        \node at (8.7,-3.7) {$-$};

         \node at (1.8,-7.2) {$\mapsto$};
         \node at (5.3,-7.2) {$+$};
         \node at (8.7,-7.2) {$-$};
         \node at (12.3,-7.2) {$-$};
        
        }
        \end{tikzpicture}
 \caption{The three skein relations defining the action of $\symm_n$ on $\CC[NCP(n)]$. The red vertices
      are adjacent indices $i, i+1$
      and the shaded blocks have at least three elements. The symmetric 3-term relation obtained
      by reflecting the middle relation across the $y$-axis is not shown.}
\label{skein relations}
\end{figure}
A more detailed description of this action can be found in \cite{Rhoades}.
Rhoades showed that we again have an $\symm_n$-equivariant linear projection $p_{\Pi}: \mathbb{C}[\Pi(n)] \rightarrow \mathbb{C}[NCP(n)]$ given for any set partition $\pi$ by 
\begin{equation}
\pi \mapsto w^{-1} \cdot (w \circ \pi),
\end{equation}
where $w$ is any permutation for which $w \circ \pi$ (here $\circ$ denotes the action of $\symm_n$ on all set partitions) is noncrossing. The kernel of this projection is generated by the set of all elements of the form $w \circ (s_i \circ \pi +\sigma(\pi))$ (the skein relations) for any permutation $w$ and almost noncrossing set partition $\pi$.

Rhoades was able to determine the $\symm_n$-irreducible structure of this module. In particular, the span of all singleton-free noncrossing set partitions with exactly $k$ blocks is an $\symm_n$-irreducible of shape $(k,k,1^{n-2k})$, and the span of all noncrossing set partitions with exactly $s$ singletons and exactly $k$ non-singleton blocks is isomorphic to an induction product of $S^{(k,k,1^{n-2k-s})}$ with the sign representation of $\symm_s$. Similarly, if we restrict the action on $\mathbb{C}[NCM(n)]$ to the span of noncrossing perfect matchings, i.e. noncrossing matchings of $[2n]$ with exactly $n$ pairs, then this action gives a model for the $\symm_n$-irreducible of shape $(n,n)$ called the $SL_2$-web basis. If we instead restrict to the span of noncrossing matchings with exactly $k$ pairs, we get a submodule isomorphic to the induction product of $S^{(k,k)}$ and a sign representation of $\symm_{n-2k}$. By the Pieri rule, this induction product is a direct sum of three irreducible submodules, one of which is isomorphic to $S^{(k,k,1^{n-2k})}$ appears in this induction product, so there exists a unique embedding of $\mathbb{C}[NCP(n)_0]$, the span of all singleton-free noncrossing set partitions in $\mathbb{C}[NCP(n)]$, into $\mathbb{C}[NCM(n)]$. The main result of this paper explicitly describes the embedding as follows:
\begin{theorem}
The linear map $f_n: \mathbb{C}[NCP(n)_0] \rightarrow \mathbb{C}[NCM(n)]$ defined by
\[
f_n(\pi) = \sum_{m \in M_\pi(n) } m
\]
is a $\symm_n$-equivariant embedding of vector spaces. Here $M_\pi(n)$ is defined to be the set of all matchings $m$ in $M(n)$ for which each block of $\pi$ contains exactly one pair in $m$.
\end{theorem}
For an example of this map, let $\pi = \{\{1,2,3\},\{4,5\}\}$ then 
\[
f_n(\pi) = \{\{1,2\},\{4,5\}\} + \{\{1,3\},\{4,5\}\} + \{\{2,3\}, \{4,5\}\}
\]
is a sum of 3 matchings in $\mathbb{C}[NCM(n)]$.

The rest of the paper is organized as follows. Section 2 will provide necessary background information. Section 3 will prove our main result, the embedding from $\mathbb{C}[NCP(n)]$ to $\mathbb{C}[NCM(n)]$. Section 4 will determine the image of this embedding within $\mathbb{C}[NCM(n)]$.

\section{Background}

\subsection{$\symm_n$-representation theory}
For $n \in \mathbb{Z}_{\geq 0}$, a {\em partition} of $n$ is a weakly decreasing sequence $\lambda = (\lambda_1, \lambda_2, \dots, \lambda_k)$ of positive integers such that $\lambda_1 + \cdots + \lambda_k = n$. Partitions of $n$ can be represented by {\em Young diagrams}, which is an arrangement of square boxes into $n$ left-justified rows, with the $i^{th}$ row containing $\lambda_i$ boxes.

 Irreducible representations of the symmetric group $\symm_n$ are naturally indexed by partitions of $n$. Let $S^\lambda$ denote the $\symm_n$-irreducible corresponding to partition $\lambda$. Given two representations $V$ and $W$  of $\symm_{m_1}$ and $\symm_{m_2}$ respectively, with $m_1 + m_2 = n$, the induction product $V \circ W$ is given by
\[
V \circ W = \mathrm{Ind}_{\symm_{m_1} \times \symm_{m_2}}^{\symm_n} V \otimes W
\]
where $\symm_{m_1} \times \symm_{m_2}$ is identified with the parabolic subgroup of $\symm_n$ which permutes $\{1, \dots, m_1\}$ and $\{m_1+1, \dots, n\}$ separately. When $V$ is an irreducible representation $S^\mu$ for some partition $\mu$ of $m_1$ and $W$ is a trivial representation of $\symm_{m_2}$, the Pieri rule describes how to express $V \circ W$ in terms of irreducibles, 
\begin{equation}
S^{\mu} \circ \mathrm{triv}_{\symm_{m_2}} \cong \sum_\lambda S^\lambda
\end{equation}
where the sum is over all partitions $\lambda$ whose young diagram can be obtained from that of $\mu$ by adding $m_2$ boxes, no two in the same column. The dual Pieri rule describes the same when $W$ is a sign representation instead of a trivial representation we again have
\begin{equation}
S^{\mu} \circ \mathrm{sign}_{\symm_{m_2}} \cong \sum_\lambda S^\lambda
\end{equation}
but the sum is now over all partitions $\lambda$ whose young diagram can be obtained from that of $\mu$ by adding $m_2$ boxes, no two in the same {\em row}.

\subsection{Projections and their kernels}
Here we provide justification for the two claims made in the introduction regarding spanning sets for the kernels of the projection maps. The first such claim is standard, while the second indirectly follows from results in \cite{Rhoades}. We include a proof of the first, the proof of the second is analogous.
\begin{proposition}
\label{kernelm}
The kernel of the projection $p_M : \mathbb{C}[M(n)] \rightarrow \mathbb{C}[NCM(n)]$ is spanned by elements of the form 
\begin{multline}
\label{kernelform}
\{\{a_1,a_2\},\{a_3,a_4\},\{a_5,a_6\}, \dots, \{a_{2k-1}, a_{2k}\}\} \\+ \{\{a_1,a_3\},\{a_2,a_4\},\{a_5,a_6\}, \dots, \{a_{2k-1}, a_{2k}\}\} \\+ \{\{a_1,a_4\},\{a_2,a_3\},\{a_5,a_6\}, \dots, \{a_{2k-1}, a_{2k}\}\}
\end{multline}
for any $a_1, \dots, a_{2k} \in [n]$, i.e. sums of three matchings which differ by a Ptolemy relation.
\end{proposition}
\begin{proof}
Let $\beta$ denote the set of all elements of the form given in equation~\ref{kernelform}. To see that the span of $\beta$ is contained in the kernel of $p_M$, note that by the $\symm_n$-equivariance of $p_M$ it suffices to check that applying $p_M$ gives 0 in the case where $a_i = i$ for all $i$. In this case, we have
\begin{multline}
p_M( \{1,2/3,4/ \cdots /2k-1,2k\} + \{1,3/2,4/ \cdots/2k-1,2k\} + \{1,4/2,3/ \cdots /2k-1,2k\}) \\ =  \{1,2/3,4/ \cdots /2k-1,2k\} + (2,3) \cdot (-\{1,2 / 3,4 / \cdots / 2k-1,2k\}) + \{1,4/2,3/ \cdots /2k-1,2k\} \\ =0
\end{multline}

To see that the kernel is contained in the span, note that since $p_M$ is a projection, the kernel is spanned by $m -p_M(m)$ for any matching $m$. Let $t$ denote the minimum number of transpositions $s_{i_1},\dots, s_{i_t}$ for which $(s_{i_1}\cdots s_{i_t})\star m$ is noncrossing, and let $w = s_{i_1}\cdots s_{i_t}$. We will show by induction on $t$ that $m- p_M(m) \in \textrm{span}(\beta)$. When $t=0$, then $m - p_M(m) = 0$, so the claim is true. Otherwise, assume the claim holds for $t-1$. We have $m - p_M(m) = s_{i_1}\circ (s_{i_1} \circ m) - s_{i_1}\cdot p_M(s_{i_1}\circ m)$. By our inductive hypothesis, $s_{i_1} \circ m - p_M(s_{i_1}\circ m)$ lies in the span of $\beta$, so it suffices to verify for any $b\in \beta$, that if we apply $s_{i_1} \circ (-)$ to every crossing term of $b$ and apply either $s_{i_1} \cdot (-)$ or $s_{i_1} \circ(-)$ to every noncrossing term of $b$, we remain in the span of $\beta$. This is true because $\beta$ is closed under the $\circ$ action, and for every noncrossing matching $m_1$, either
\[
s_{i_1} \circ m = s_{i_1} \cdot m
\]
or
\[
s_{i_1} \cdot m_1 = s_{i_1} \circ m_1 - (s_{i_1}\circ m_1 + m_1 + m_1')
\] where $m_1'$ is obtained by replacing the sets $\{i,a\}$ and $\{i+1,b\}$ with the sets $\{i,i+1\}$ and $\{a,b\}$. In the second case, $s_{i_1}\circ m_1 + m_1 + m_1'$ is in $\beta$.
\end{proof}

\begin{proposition}
\label{kernelp}
The kernel of the projection $p_\Pi : \mathbb{C}[\Pi(n)] \rightarrow \mathbb{C}[NCP(n)]$ is spanned by elements of the form 
\begin{align*}
w \circ (s_i \circ \pi +\sigma(\pi))
\end{align*} for any permutation $w$ and singleton-free almost noncrossing set partition $\pi$, i.e. sums of set partitions which differ by a skein relation.
\end{proposition}

\section{The embedding} In order to prove that our map is an embedding, it will be helpful to introduce a multiplicative structure to work with. To do so we will introduce three graded commutative $\mathbb{C}$-algebras $R_n$, $A_n$, and $M_n$, all with $\symm_n$-actions. If we forget the multiplicative structure, the underlying $\symm_n$-modules of $R_n$, $A_n$, and $M_n$ will contain a copy $\mathbb{C}[\Pi(n)]$, $\mathbb{C}[M(n)]$, and $\mathbb{C}[NCM(n)]$ respectively. In the case of $M_n$, this copy will be all of $M_n$. The structure of this proof is best explained via a commutative diagram, see Figure~\ref{cdasd}. 
We will define a map $h_n \circ \iota_\Pi: \mathbb{C}[\Pi(n)_0] \rightarrow M_n$, and show that its kernel  is equal to the kernel of $p_\Pi$. The desired embedding $f_n$ then follows from the first isomorphism theorem.

\begin{figure}[h]
\label{cdasd}
\begin{tikzcd}
R_n \arrow[r, "g_n"]\arrow[rr, "h_n", bend left] & A_n \arrow[r, two heads, "q"]  & M_n \arrow[d, leftrightarrow, "\cong"]\\
& \mathbb{C}[M(n)] \arrow[r, two heads, "p_M"]\arrow[u, hook, "\iota_M"] & \mathbb{C}[NCM(n)]\\
\mathbb{C}[\Pi(n)_0] \arrow[uu, hook, "\iota_\Pi"]\arrow[rr, two heads, "p_\Pi"] & &\mathbb{C}[NCP(n)_0]\arrow[u, dashed, hook, "f_n"]
\end{tikzcd}
\caption{A commutative diagram of the maps used in the following proofs. All maps shown are $\symm_n$-equivariant linear maps. Maps between $R_n$, $A_n$, and $M_n$ are also morphisms of $\mathbb{C}$-algebras. The desired embedding is shown as a dashed arrow.}
\end{figure}
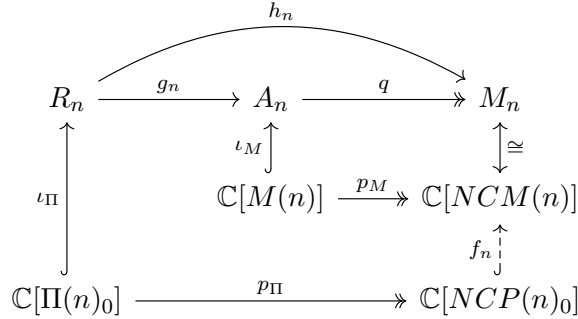

We begin with the definition of $R_n$.

\begin{definition} Let $n \in \mathbb{N}$. Define $R_n$ to be the unital graded commutative $\mathbb{C}$-algebra generated by nonempty subsets of $[n]$ where each generator has degree $1$.  Define a degree-preserving action of $\symm_n$ on $R_n$ by
\[
\pi \cdot \{a_1, \dots, a_k\} = \mathrm{sign}(\pi)\{\pi(a_1), \cdots, \pi(a_k)\}
\]
 for any permutation $\pi \in \symm_n$ and generator $\{a_1, \cdots, a_k\} \in R_n$.
\end{definition}

The ring $R_n$ can be thought of as the ring of multiset collections of subsets of $[n]$ with multiplication given by union of collections and addition purely formal. It is in this sense that it contains a copy of $\mathbb{C}[\Pi(n)]$, as set partitions of $n$ are particular collections of subsets of $[n]$. To be precise, there exists an $\symm_n$-module embedding $\iota_\Pi: \mathbb{C}[\Pi(n)_0] \hookrightarrow R_n$, given by sending any singleton-free set partiton $\pi$ to the product of its blocks.

The ring $A_n$ is a subring of $R_n$ designed to model matchings in much the same way which $R_n$ models set partitions. It is defined as follows.
\begin{definition}
Let $n \in \mathbb{N}$ and define $A_n$ to be the $\symm_n$-invariant subalgebra of $R_n$ generated by the size two subsets of $[n]$. The subring $A_n$ is invariant under the $\symm_n$-action of $R_n$, and thus inherits a graded $\symm_n$-action from $R_n$.
\end{definition}

Like $R_n$, the ring $A_n$ can be thought of as the ring of multiset collections of size-two subsets of $[n]$. As matchings are particular collections of size-two subsets of $[n]$, we again have an $\symm_n$-module embedding $\iota_M:\mathbb{C}[M(n)] \hookrightarrow A_n$,  given by 
\[
\{\{a_1,b_1\}, \dots, \{a_k,b_k\}\} \longmapsto \{a_1,b_1\} \cdots  \{a_k,b_k\}
\]
for any matching $\{\{a_1,b_1\},  \dots, \{a_k,b_k\}\})$.

Our final ring, $M_n$, is defined as a quotient of $A_n$ in the following way.
\begin{definition}
Define $I_n$ to be the ideal of $A_n$ generated by elements of the following forms
\begin{itemize}
\item $\{a,b\} \cdot \{a,b\}$\\
\item $\{a,b\}\cdot\{a,c\}$\\
\item $\{a,b\}\cdot\{c,d\} + \{a,c\}\cdot\{b,d\} + \{a,d\}\cdot\{b,c\}$
\end{itemize}
for any distinct $a,b,c,d \in [n]$. Then $I_n$ is a homogeneous $\symm_n$-invariant ideal of $A_n$, so define $M_n$ to be the graded $\symm_n$-module  $M_n := A_n/I_n$. Let $q : A_n \rightarrow M_n$ be the quotient map.
\end{definition}

The first two types of elements listed in the definition of $I_n$ serve the purpose of removing collections of size-two subsets which are not actually matchings. The third is the Ptolemy relation used to define the action of $\symm_n$ on $\mathbb{C}[NCM(n)]$, so quotienting by this ideal gives an $\symm_n$-module isomorphic to $\mathbb{C}[NCM(n)]$, as per the following argument.
\begin{proposition}
\label{iso}
There is an $\symm_n$-module isomorphism from $\mathbb{C}[NCM(n)]$ to $M_n$, given by 
\[
\{\{a_1,b_1\}, \dots, \{a_k,b_k\}\} \longmapsto \{a_1,b_1\} \cdots  \{a_k,b_k\}
\]
for any noncrossing matching $\{\{a_1,b_1\},  \dots, \{a_k,b_k\}\}$.
\end{proposition}
\begin{proof}
Let $q: A_n \rightarrow M_n$ be the quotient map. Consider the map $q \circ \iota_M : \mathbb{C}[M(n)] \rightarrow M_n$. The kernel of $q \circ \iota_M$ is spanned by the elements of $\mathbb{C}[M(n)]$ which are sent under $\iota_M$ to a multiple of $\{a,b\}\cdot\{c,d\} + \{a,c\}\cdot\{b,d\} + \{a,d\}\cdot\{b,c\}$ for some distinct $a,b,c,d \in [n]$. Therefore the kernel of $q \circ \iota_M$ is equal to the kernel of $p_M$ by Proposition~\ref{kernelm}. The image of $q\circ \iota_M$ is all of $M_n$. To see this, note that products of generators of $A_n$ form a vector space basis for $A_n$, and every such basis element is either in the image of $\iota_M$ or in $I_n$.  We therefore have
\begin{equation}
\mathbb{C}[NCM(n)] \cong \mathbb{C}[M(n)]/\mathrm{ker}(p_M) = \mathbb{C}[M(n)]/\mathrm{ker}(q \circ \iota_M) \cong \mathrm{im}(q \circ \iota_M) = M_n 
\end{equation}
where the isomorphism on the left is induced by the map $p_M$ and the isomorphism on the right is induced by the map $q \circ \iota_M$. Composing these isomorphisms gives the stated map.
\end{proof}

The following definition is the key idea behind our main theorem.
\begin{definition}
Let $n\in \mathbb{N}$. Define the $\mathbb{C}$-algebra map $g_n: R_n \rightarrow A_n$  by
\[
g_n(A) = \sum_{\{a,b\} \subseteq A} \{a,b\}
\]
for generators $A \in R_n$. Singleton sets are sent to $0$ by $g_n$. Define $h_n := q \circ g_n$ where $q$ is the quotient map $A_n \rightarrow M_n$.
\end{definition}
We give the definition in terms of $R_n$, $A_n$, and $M_n$ for simplicity and ease of proofs later, but the map we really care about is $ h_n \iota_\Pi : \mathbb{C}[\Pi(n)] \rightarrow M_n$. Under this map, a set partition $\pi$ is sent to the product of its blocks, then each block is sent to the sum of all size-two subsets it contains. After distributing, we get a sum of all ways to pick a size two subset from each block. Composing with the isomorphism between $M_n$ and $\mathbb{C}[NCM(n)]$ we get the sum of all matchings such that each block of $\pi$ contains exactly one pair of the matching, as in Theorem~\ref{maintheorem}.

 The $\symm_n$-equivariance of $h_n$ is simple to check. We have the  following calculation.
\begin{lemma}
\label{kernel}
Let $i,j \geq 2$ and let $p_1, p_2, \dots, p_i$ and $q_1, q_2, \dots, q_j$ be in $[n]$. Then the element of $R_n$ 
\begin{align*}
\kappa_n := \{p_1,\dots, p_i\}\cdot \{q_1, \dots, q_j\} &- \{p_1, \cdots p_{i-1}\} \cdot \{q_1, \cdots q_{j}, p_{i}\} -  \{p_1, \cdots p_{i}, q_j\} \cdot \{q_1, \cdots q_{j-1}\} \\&+  \{p_1, \cdots p_{i-1}, q_j\} \cdot \{q_1, \cdots q_{j-1}, p_{i}\}  +  \{p_1, \cdots p_{i-1}, q_1, \cdots q_{j-1}\} \cdot \{p_i, q_j\}
\end{align*}
lies in the kernel of $h_n$.
\end{lemma}

When $i,j >2$, the element $\kappa_n$ corresponds to the five-term skein relation depicted in Figure 1. If either $i$ eqauls 2, then $\{p_1, \cdots p_{i-1}\} = \{p_1\}$ is a one element set and therefore sent to 0 by $h_n$, removing the term containing $\{p_1\}$ corresponds to the four-term skein relation depicted in Figure 1. Similarly, if $j$ equals 2 or $i$ and $j$ both equal two, removing the terms in $\kappa_n$ which are individually sent to 0 corresponds to the four or three-term skein relation depicted in Figure 1.

\begin{proof}
Applying $h_n$ to $\kappa_n$ gives
\begin{align*}
\sum_{\substack{\{a,b\} \subseteq \{p_1,\dots, p_i\}\\\{c,d\} \subseteq \{q_1, \dots, q_j\}}} \{a,b\}\cdot \{c,d\} &- \sum_{\substack{\{a,b\} \subseteq \{p_1,\dots, p_{i-1}\}\\\{c,d\} \subseteq \{q_1, \dots, q_j, p_i\}}} \{a,b\}\cdot \{c,d\} - \sum_{\substack{\{a,b\} \subseteq \{p_1,\dots, p_i, q_j\}\\\{c,d\} \subseteq \{q_1, \dots, q_{j-1}\}}} \{a,b\}\cdot \{c,d\} \\&+ \sum_{\substack{\{a,b\} \subseteq \{p_1,\dots, p_{i-1}, q_j\}\\\{c,d\} \subseteq \{q_1, \dots, q_{j-1}, p_i\}}} \{a,b\}\cdot \{c,d\} + \sum_{\substack{\{a,b\} \subseteq \{p_1,\dots, p_{i-1}, q_1, \cdots, q_{j-1}\}\\\{c,d\} \subseteq \{p_i,q_j\}}} \{a,b\}\cdot \{c,d\}
\end{align*}
Note that the pairs of sets defining the first and second summations in the above expression differ only in the location of $p_i$, and similarly for the third and fourth. Since these summations come with opposite signs, the $\{a,b\}, \{c,d\}$  terms in the above expression will cancel unless one of $a,b,c,d$ is equal to $p_i$. Similarly, comparing the first and third sums and the second and fourth sums we find cancellation unless at least one of $a,b,c,d$ is equal to $q_j$. If the remaining two elements of $a,b,c,d$ are both $p$'s or both $q$'s, then $\{a,b\}\cdot \{c,d\}$ also cancels. Therefore we have
\begin{equation}
h_n(\kappa_n) = \sum_{\substack{a \in \{p_1, \dots, p_{i-1}\} \\ b\in \{q_1, \dots, q_{j-1}\}}} \{a, p_i\} \cdot \{b,q_j\} + \{a, q_j\} \cdot \{b, p_i\} + \{a,b\}\cdot \{p_i,q_j\}
\end{equation}
which is manifestly a sum of the definining relations of $M_n$.
\end{proof}

The above calculation allows for the identification of the kernel of $h\circ \iota_\Pi$ and $p_\Pi$ mentioned in the beginning of this section.
\begin{proposition}
\label{hiota} The kernel of $h \circ \iota_\Pi$ is generated by the set of all elements of the form $w \circ (s_i \circ \pi +\sigma(\pi))$ (the skein relations) for any permutation $w$ and singleton-free almost noncrossing set partition $\pi$.
\end{proposition}

\begin{proof}
By Lemma~\ref{kernel}, all such elements lie in the kernel. Since restricting $\mathbb{C}[\Pi(n)_0]$ to noncrossing set partitions with exactly $k$ blocks gives an irreducible submodule \cite[Proposition~5.2]{Rhoades}, and each such irreducible is sent to degree $k$ in $M_n$, it suffices to demonstrate that each irreducible is not sent identically to 0. But for each noncrossing set partition $\pi \in \Pi(n)_0$, $h_n\circ \iota (\pi)$ is a sum of noncrossing matchings of $[n]$, which form a basis for $M_n$ by Proposition~\ref{iso}, and is therefore nonzero.
\end{proof}

We can now prove our main result.
\begin{theorem}
\label{maintheorem}
The linear map $f_n: \mathbb{C}[NCP(n)_0] \rightarrow \mathbb{C}[NCM(n)]$ defined by
\[
f_n(\pi) = \sum_{m \in M_\pi(n) } m
\]
is a $\symm_n$-equivariant embedding of vector spaces. Here $M_\pi(n)$ is defined to be the set of all matchings $m$ in $M(n)$ for which each block of $\pi$ contains exactly one pair in $m$.
\end{theorem}
\begin{proof}
By Proposition~\ref{hiota} and Proposition~\ref{kernelp}, the kernel of $h \circ \iota_\Pi$ is equal to the kernel of $p_\Pi$. So we have
\begin{equation}
\mathbb{C}[NCP(n)_0] \cong \mathbb{C}[\Pi_0(n)]/ \mathrm{ker}(p_\Pi) = \mathbb{C}[\Pi(n)_0]/\mathrm{ker}(h\circ \iota_\Pi) \cong \mathrm{im}(h\circ \iota_\Pi) \subset M_n \cong \mathbb{C}[NCM(n)]  
\end{equation}
where the isomorphism on the left is induced by $p_\Pi$ and the isomorphism on the right is induced by $h \circ \iota_\Pi$. Chasing these isomorphisms and inclusions results in the map $f_n$.
\end{proof}

\section{The image}
We have an embedding $f_n: \mathbb{C}[NCP(n)_0]  \hookrightarrow \mathbb{C}[NCM(n)]$, so it is a natural question to ask for a description of the image of $f_n$ within $\mathbb{C}[NCM(n)]$. Via the commutative diagram in Figure 2,we have an isomorphism of images
\begin{equation}
\mathrm{im}(h_n) \cong \mathrm{im}(f_n).
\end{equation}
So it is equivalent to describe the image of $h_n$, and the multiplicative structure of $M_n$ will make describing the image of $h_n$ easier. This section will show that the image of $h_n$ has a simple description, the proof of which will require the following lemmas.

\begin{lemma}
\label{squarecancel}
Let $A \subseteq [n]$. Then $h_n(A)^2 = 0$. 
\end{lemma}
\begin{proof}
Applying the definition of $h_n$ gives
\begin{equation}
h_n(A)^2 = \sum_{\substack{a,b \in [n] \\ a \neq b}} \sum_{\substack{c,d \in [n] \\ c \neq d}} \{a,b\} \cdot \{c,d\} 
\end{equation}
Using the defining relation of $M_n$ that
\[
\{a,b\} \cdot \{a,c\} = 0
\]
we have
\begin{equation}
\sum_{\substack{a,b \in [n] \\ a \neq b}} \sum_{\substack{c,d \in [n] \\ c \neq d}} \{a,b\} \cdot \{c,d\} =\frac{1}{3} \sum_{\substack{a,b,c,d \in [n]\\ a,b,c,d \,\mathrm{ distinct}}} \{a,b\} \cdot \{c,d\} + \{a,c\} \cdot \{b,d\} + \{a,d\} \cdot \{b,c\}.
\end{equation}
The right hand side of the above equation equals 0 because
\[
\{a,b\} \cdot \{c,d\} + \{a,c\} \cdot \{b,d\} + \{a,d\} \cdot \{b,c\} = 0
\]
for any distinct $a,b,c,d \in [n]$.
\end{proof}
\begin{lemma}
\label{twosetcancel}
Let $A,B$ be disjoint subsets of $[n]$. Then 
\[
h_n(A)\cdot\left(\sum_{\substack{a \in A \\ b\in B}} \{a,b\}\right) = 0.
\]
\end{lemma}
\begin{proof}
Applying $h_n$ gives
\[
h_n(A)\cdot\left(\sum_{\substack{a \in A \\ b\in B}} \{a,b\}\right) = \frac{1}{3}\sum_{a_1, a_2, a_3 \in A} \sum_{b \in B} \{a_1, a_2\} \cdot \{a_3, b\} + \{a_1,a_3\} \cdot \{a_2, b\} + \{a_2, a_3\} \cdot \{a_1,b\} = 0
\]
\end{proof}

\begin{lemma}
\label{lemma}
Let $B_1, \dots, B_k$ be the blocks of a singleton free set partition of $[n]$. Then
\[
h_n\left(\prod_{i=1}^k B_k\right) = h_n\left([n]\cdot \prod_{i=1}^{k-1} B_k\right)
\]
\end{lemma}

\begin{proof}
We have the following calculation:
\begin{align*}
h_n\left([n] \cdot \prod_{i=1}^{k-1} B_i\right) &= \left( \sum_{\substack{a,b \in [n] \\ a\neq b}} \{a,b\}\right) \cdot  h_n\left( \prod_{i=1}^{k-1} B_i\right) \\ &= \left(\left(\sum_{i=1}^k h_n(B_i) \right) + \left(\sum_{1\leq i < j \leq k} \sum_{\substack{a \in B_i \\ b \in B_j}} \{a,b\}\right)\right) \cdot h_n\left(\cdot \prod_{i=1}^{k-1} B_k\right) \\ &= h_n(B_i) \cdot \left(\cdot \prod_{i=1}^{k-1} h_n(B_i)\right)
\end{align*}
where the last line follows by the preceeding two lemmas, as every term in the outer sum of
\[
\sum_{1\leq i < j \leq k} \sum_{\substack{a \in B_i \\ b \in B_j}} \{a,b\}
\]
is annihilated by some term in the product
\[
\prod_{i=1}^{k-1} h_n(B_k)
\]
and every term except the $i=k$ term in the sum
\[
\sum_{i=1}^k h_n(B_i)
\]
is annihilated by some term in the product
\[
\prod_{i=1}^{k-1} h_n(B_k).
\]
\end{proof}

We can now describe the image of $h_n$.

\begin{theorem}
 Let $H_n$ be the ideal of $M_n$ generated by $h_n([n])$. Then
\[
\mathrm{im}(h_n) = H_n.
\]
\end{theorem}

\begin{proof}
It is immediate from Lemma~\ref{lemma} that the image of $h_n$ is contained in $H_n$, so it suffices to show that
\begin{equation}
\mathrm{dim}(H_n) \leq \mathrm{dim}(\mathrm{im}(h_n)) = \mathrm{dim}(\mathrm{im}(f_n)) = \#NCP(n)_0
\end{equation}
To bound the dimension of $H_n$, note that for any fixed $a \in [n]$, 
\[
h_n([n])\cdot\left( \sum_{\substack{b \in [n] \\b \neq a}} \{a,b\}\right) = \frac{1}{3} \sum_{\substack{b \in [n] \\b \neq a}}\sum_{\substack{c \in [n] \\c \neq a}}\sum_{\substack{d \in [n] \\d \neq a}} (\{a,b\}\cdot\{c,d\} + \{a,c\}\cdot \{b,d\} + \{a,d\}\cdot\{b,c\}) = 0
\]
so
\[
h_n([n])\cdot \{1,a\} = -h_n([n])\cdot \left( \sum_{\substack{b \in [n] \\b \neq a,1}} \{a,b\} \right).
\]
Therefore $H_n$ is spanned by elements of the form
\[
h_n([n]) \cdot \{a_1, a_2\} \cdot \{a_3, a_4\} \cdots \{a_{2k-1}, a_{2k}\}
\]
where the sets $\{a_1, a_2\}, \dots, \{a_{2k-1}, a_{2k}\}$ form a noncrossing matching of $\{2,\dots, n\}$. These elements are not linearly independent, however. Consider the element $\tilde{f}_n(\tilde{\pi})$ of $M_n$ given by 
\[
\tilde{f}_n(\tilde{\pi}) := \prod_{B \in \tilde{\pi}} h_n(B)
\]
 for any singleton free noncrossing set partition $\tilde{\pi}$ of $\{2, \dots, n\}$ (i.e. the ``image of $f_n$" if such a set partition were in the domain of $f_n$). Let $B_1$ be the block of $\tilde{\pi}$ containing $2$, and let $\pi$ be the set partition of $[n]$ obtained by adding $1$ to block $B_1$. We have
\begin{align*}
h_n([n]) \cdot \tilde{f}_n(\tilde{\pi}) &= h_n([n])\cdot \prod_{B \in \tilde{\pi}} h_n(B) \\&= h_n(B_1) \cdot h_n \left([n] \cdot \prod_{B\neq B_1 \in \pi} B\right) \\&= h_n(B_1) \cdot h_n\left(\prod_{B\in \pi} B\right) \\&= h_n(B_1) h_n(B_1 \cup \{1\}) h_n \left(\prod_{B\neq B_1 \in \pi} B\right)\\&=0
\end{align*}
The third equality follows from Lemma~\ref{lemma} and the final equality follows from the fact that
\[
h_n(B_1) h_n(B_1 \cup \{1\}) = h_n(B_1)^2 + h_n(B_1) \left( \sum_{ b \in B_1} \{1,b\} \right) = 0
\]
which follows from Lemma~\ref{twosetcancel} and Lemma~\ref{squarecancel}.
The collection of $\tilde{f}_n(\tilde{\pi})$ for singleton-free noncrossing set partitions $\pi$ of $\{2, \dots, n\}$ is linearly independent. To see this, note that any linear relation among the $\tilde{f}_n(\tilde{\pi})$ would also be a linear relation among $f_{n-1}(\pi)$ where $\pi$ is the set partition of $[n-1]$ obtained by decrementing the indices in $\tilde{\pi}$. But $f_{n-1}$ is an embedding and singleton-free noncrossing set partitions are linearly independent in $\mathbb{C}[NCP(n-1)_0]$. The dimension of $H_n$ is therefore bounded by
\begin{multline}
\mathrm{dim} (H_n) \leq \# \{ \textrm{noncrossing matchings of }\{2,\dots, n\} \} \\- \# \{ \textrm{singleton-free noncrossing set partitions of }\{2,\dots, n\} \}.
\end{multline}
Noncrossing matchings of $\{2, \dots, n\}$ are in bijection with noncrossing set partitions of $[n]$ in which only the block containing $1$ may be a singleton (though it may be larger). Given a noncrossing set partition, simply take the matching that matches the largest and smallest element of each block not containing $1$. Sinlgeton-free noncrossing set partitions of $\{2, \dots, n\}$ are in bijection with set partitions of $[n]$ in which $\{1\}$ is the unique singleton block. We therefore have
\[
\mathrm{dim}(H_n) \leq\#NCP(n)_0
\]
 as desired.
\end{proof}
\section{Future directions}
One of the goals motivating this paper is to find new combinatorially nice bases for $\symm_n$-irreducibles which arise from existing bases in an analogous way to the skein action. More specifically, suppose we have a basis for $S^{\lambda}$ which is indexed by certain structures on the set $[k]$, where $k = |\lambda|$ (e.g. noncrossing perfect matchings, in the case of this paper).  We can create a basis for the induction product of  $S^{\lambda}$ with a sign representation of $\symm_{n-k}$ indexed by all ways to put a certain structure on a $k$-element subset of $[n]$. The Pieri rule tells us which $\symm_n$ irreducibles this decomposes into. In particular, there will be one copy of $(\lambda, 1^{n-k})$. How do we isolate that irreducible?

It is optimistic to think that there will be a method that works in any sort of generality, but perhaps analogs could be found in certain specific cases. For example, an analog might exist for the $A_2$-web basis for $S^{(k,k,k)}$ introduced by Kuperberg \cite{k}. The web basis consists of planar bipartite graphs embedded in a disk with $n$ boundary vertices all of degree 1, interior vertices are degree 3, all boundary vertices are part of the same bipartition, and no cycles of length less than 6 exist. One potential candidate for a basis for $S^{(k,k,k,1^{n-3k})}$ is as follows.

\begin{conjecture}
Let $A$ be the set of all planar bipartite graphs embedded in a disk for which the following conditions hold
\begin{itemize}
\item There are $n$ vertices on the boundary of the disk, and there exists a bipartition in which all of these vertices are in the same part.
\item Every interior vertex is connected to a boundary vertex
\item Every interior vertex in the same bipartition as the boundary vertices is degree 3. These are called negative interior vertices
\item Every interior vertex not in the same bipartition as the boundary vertices is degree at least 3. These are called positive interior vertices.
\item The number of positive interior vertices minus the number of negative interior vertices is exactly $k$.
\item No cycles of length less than 6 exist.
\end{itemize}
Then $|A|$ is equal to the dimension of $S^{(k,k,k,1^{n-3k})}$.
\end{conjecture}

The set $A$ can be thought of as consisting of webs for which the condition of interior vertices being degree 3 has been partially relaxed. The conjecture can be shown to hold for $k=2$ and any $n$, as well as $n=10, k=3$ via direct enumeration. If the above conjecture is true, it suggests the following question.

\begin{question}
Does there exist a combinatorially nice action of $\symm_n$ on $\mathbb{C}[A]$ which creates a $\symm_n$ module isomorphic to $S^{(k,k,k,1^{n-3k})}$? If so, what does the unique embedding into $S^{(k,k,k)}$ induced with a sign representation of $\symm_{n-3k}$ look like?
\end{question}

A positive answer to this question might help elucidate how to apply similar methods more generally.

\section{Acknowledgements}
We are very grateful to Brendon Rhoades for many helpful discussions and comments on this project.

\end{document}